\documentclass[12pt]{amsart}

\newtheorem{theorem}{Theorem}[section]

\newtheorem{proposition}[theorem]{Proposition}

\newtheorem{corollary}[theorem]{Corollary}

\newtheorem{remark}[theorem]{Remark}

\newtheorem{definition}[theorem]{Definition}

\newtheorem{conjecture}[theorem]{Conjecture}

\usepackage{amssymb}

\numberwithin{equation}{section}

\begin{document}

\baselineskip=16pt

\title[Chow stability]{On Chow stability for algebraic curves}

\author{L. Brambila-Paz}

\address{CIMAT, Apdo. Postal 402, C.P. 36240. Guanajuato, Gto,
M\'exico}
\email{lebp@cimat.mx}

\author{H. Torres-L\'opez}

\address{CIMAT, Apdo. Postal 402, C.P. 36240. Guanajuato, Gto,
M\'exico} \email{hugo@cimat.mx}

\keywords{Chow stability, tangent bundle, stability, projective
varieties, Hilbert scheme}

\subjclass[2010]{14H60, 14H10, 14C05, 14D23}

\date{\today}

\begin{abstract} In the last decades there have been introduced different concepts of stability for projective varieties.
In this paper we give a natural and intrinsic criterion of the Chow, and Hilbert, stability for complex irreducible smooth projective curves $C\subset \mathbb P ^n$.
 Namely, if the restriction $T\mathbb P_{|C} ^n$ of the tangent bundle of $\mathbb P ^n$ to $C$ is stable then $C\subset \mathbb P ^n$ is Chow stable, and hence Hilbert stable. We apply this criterion to describe a smooth open set of the irreducible component $Hilb^{P(t),s}_{{Ch}}$ of the Hilbert scheme of $\mathbb{P} ^n$
containing the generic smooth Chow-stable curve of genus $g\geq 4$ and degree $d>g+n-\left\lfloor\frac{g}{n+1}\right\rfloor.$
Moreover, we describe the quotient stack of such curves. Similar results are obtained for the locus of Hilbert stable curves.
\end{abstract}

\maketitle

\section{Introduction}\label{intro}

 In \cite{mum} Mumford introduced the GIT notion of Chow stability
 (see Definition \ref{defchow}) giving projective moduli spaces of
projective varieties.  However, in general there
is no simple way to know when a variety is Chow stable, mainly,
because the Hilbert-Mumford criterion has not been successfully
simplified or interpreted for varieties of higher dimension.  Many
authors  have turned to other methods and have introduced different
concepts of stability  for producing moduli of varieties
(see e.g. \cite{alex}, \cite{ko1}, \cite{rich} ). Some of
them were defined with the aim to understand the relation of  the
algebro-geometric stability and the existence of special metrics.
It was R. Berman who,  in \cite{ber} (see also \cite{tian}), proved that a Fano
manifold admitting  a K\"ahler-Einstein metric is $K$-polystable.
The breakthrough result has been achieved recently by Xiu-Xiong
Chen, Simon Donalson and Song Sun in \cite{donsun}. They showed
that if a Fano manifold is $K$-stable then it admits a
K\"ahler-Einstein metric.
 For more information in this direction see \cite{donsun} and \cite{oda} and the bibliography therein.

In this paper, for complex irreducible  smooth curves, we prove in Section $2$

\begin{theorem}\label{teotan} Let $C \subset \mathbb{P}^n$ be a
complex irreducible smooth  curve. If the restriction $T\mathbb P_{|C} ^n$ of the
tangent bundle of $\mathbb P ^n$ to $C$ is semistable  then $C\subset
\mathbb P ^n$ is Chow semistable. Moreover, if $T\mathbb P_{|C} ^n$ is stable then
$C\subset \mathbb P ^n$ is Chow stable.
 \end{theorem}

Another way of stating Theorem \ref{teotan}, via the Hitchin-Kobayashi correspondence, is:

\begin{theorem} Let $C \subset \mathbb{P}^n$ be a
complex irreducible smooth curve. If  $T\mathbb P_{|C} ^n$ admits
an Hermitian-Einstein metric then $C \subset \mathbb{P}^n$ is Chow
poly-stable, and Chow stable if $T\mathbb P_{|C} ^n$ is irreducible.
\end{theorem}

Our theorem provides a sufficient condition for
the Chow stability for irreducible smooth curves. The proof of the
theorem is not complicated, mainly we use \cite[Theorem
4.12]{mum} and the relation between the tangent bundle
and the syzygy bundle. But our statement is not in the literature
and our main contribution is the interpretation of Chow (and Hilbert) stability
via the stability of the restriction of the tangent bundle of the
projective space to $C$. One may conjecture that Theorem
\ref{teotan} holds also for varieties of higher dimension, that is,

\begin{conjecture} Let $X \subset \mathbb{P}^n$ be a complex
irreducible smooth variety. If the restriction $T\mathbb P_{|X} ^n$ of the
tangent bundle of $\mathbb P ^n$ to $X$
 is $\mathcal{O}_X(1)$-stable  then $X\subset
\mathbb P ^n$ is Chow  stable.
\end{conjecture}

The theorem is still true if we drop the assumption
that the base field is $\mathbb{C}$. The proof works for any field $K=\bar{K}$
but we will stay with $\mathbb{C}$
since some of our applications work only for curves in $\mathbb{C}.$

Our viewpoint also sheds some new light on the singularities of projective varieties.
Suppose that $ X\subset \mathbb{P}^n$ is a complex
irreducible non-degenerate variety of dimension $r$ and $p\in X$ is a singularity of multiplicity $\mu _p.$
In Proposition \ref{proplss} we prove (see Definition \ref{dells} for the definition of linearly stable)

\bigskip

{\it \ \ \ \  \ \ \  If
$ X\subset \mathbb{P}^n$ is linearly stable (resp. semistable) then $\mu _p < \frac{\deg X}{n+1-r}$  (resp. $\leq $) for any $p\in X.$}

\bigskip

It is well known (see \cite[Theorem 4.15]{mum}), that if $C
\subset \mathbb{P}^n$ is a smooth irreducible curve of genus $g\geq 1$
embedded by a complete linear system of degree $d>2g$ then $C$ is
Chow stable. For non complete linear systems and lower degrees the existence of
Chow stable (semistable) curves is established by our next result (see Corollary \ref{corram} and \ref{teolow}). First recall that for a general curve of genus $g>0$, a sharp lower
bound for the existence of generated linear series $(L,V)$ of type
$(d,n+1)$ is $d\ge g+n-\left\lfloor\frac{g}{n+1}\right\rfloor$.
That is, the Brill-Noether number, for line bundles, $\rho (g,d,n+1):=
g-(n+1)(n-d+g)$ is non-negative (see Section $3$ for the definition of a \emph{Petri} curve).

\begin{theorem}\label{teostab0} Let $C\subset \mathbb{P}^n$ be an
irreducible smooth curve of genus $g\geq 1$
embedded by the linear series $(L,V)$ of type $(d,n+1)$. The canonical embedding of a non-hyperelliptic curve is Chow stable.
 If $C$ and $(L,V)$ are
general and $\rho (g,d,n+1)\geq 0$  then $C\subset \mathbb{P}^n$ is Chow semistable.
Moreover, $C\subset \mathbb{P}^n$ is Chow stable if one of the following conditions
\begin{enumerate}
\item $C$ is a general curve of genus $g\geq 2$ and $gcd(d,n)=1;$
\item $C$ is a curve of genus $g=1, d\geq n+1$ and $gcd(d,n)=1;$
\item $C$ is a curve of genus $g=2, d\geq n+2$ with $d\ne 2n ;$
\item $C$ is a Petri curve of genus $g\geq 3$ and $n\leq 4;$
\item $C$ is a Petri curve of genus $g\geq 2(n-2)$ and $n \geq 5;$
\end{enumerate}
is satisfied.
\end{theorem}

 Mainly, Theorem \ref{teostab0}  summarize
the known results on the stability of the
syzygy bundle. The breakthrough result in this
direction has been achieved recently in
\cite{bbn}, where D.C. Butler's Conjecture (see \cite[Conjecture 2]{but})
was proved for line bundles on smooth curves in some generality.

Let $Hilb_{P(t)}{{\mathbb{P}^n}}$ be the
Hilbert scheme of the projective space $\mathbb{P}^n$ with Hilbert polynomial $P(t)=dt+(1-g)$. The problem of describing the
Hilbert scheme parameterizing embedded curves seems very natural and interesting too.
In general, the Hilbert schemes $Hilb_{P(t)}{{\mathbb{P}^n}}$ can be quite pathological.

Denote by
$Hilb^{P(t),s}_{Ch}$ (respectively $Hilb^{P(t),ss}_{Ch}$) the irreducible component of the
Hilbert scheme $Hilb_{P(t)}{{\mathbb{P}^n}}$
containing the generic Chow stable (respectively Chow semistable) curve. The principal significance of Theorem \ref{teostab0} is that it
allows us to describe an open smooth subscheme of $Hilb^{P(t),s}_{Ch}$ (respectively $Hilb^{P(t),ss}_{Ch}$).

Recall that for $d=m(2g-2)$ with $m\geq 5,$
Mumford uses the $m$-canonical embedding $C\subset \mathbb{P}(H^0(C,K^m_C)^*)$ to construct from $Hilb^{P(t),s}_{Ch}$ the moduli space $M_g$ of smooth curves.  In our case we allow different embedding of the same curve and the embedding can be by non-complete linear systems and of degree $g+n-\left\lfloor\frac{g}{n+1}\right\rfloor\leq d$. Recall that if $d< g+n-\left\lfloor\frac{g}{n+1}\right\rfloor$, there are no general or Petri curves in $Hilb_{P(t)}{{\mathbb{P}^n}}$.

In order to state our results we recall from \cite[Chapter
XXI]{arb2} that given the universal family $p_0:\mathcal{C}\to M^0_g$, of smooth curves parameterized by  the fine moduli space $M^0_g$  of automorphisms-free smooth curves,
 there exists a  relative linear series
$\mathcal{G}_{d}^{n}(p)$ over $M^0_g$.
In Theorem \ref{teoprinbund} we recall the proof of the existence of $\mathcal{G}_{d}^{n}(p_0)$. Since
  the family $p_0:\mathcal{C}\to M^0_g$ has no section, there is no universal family. However, we construct a $ PGL_n-$principal bundle ( see Theorem \ref{teoprinbund})  $\mathcal{B}_{d}^{n}\to \mathcal{P}_d^n$ over an open irreducible subscheme $\mathcal{P}_d^n$ of $ \mathcal{G}_{d}^{n}(p_0)$ with universal properties. We use the $PGL_n$-principal bundle $\mathcal{B}_{d}^{n}$ to define a
 natural algebraic morphism $$\Gamma :\mathcal{B}_{d}^{n} \longrightarrow  Hilb^{P(t)}_{{\mathbb{P}^n}}.$$

We can now formulate our main results (see Theorems \ref{teochow11} and \ref{teochowprin1} and Corollary \ref{corred1}).

\begin{theorem}\label{teochow1} Let $g,d$ and $n$ be positive integers with
$g\geq 4$. For any  $d\geq g+n-\left\lfloor\frac{g}{n+1}\right\rfloor, $ $Hilb^{P(t),ss}_{Ch}\ne \emptyset .$ Moreover, if $d> g+n-\left\lfloor\frac{g}{n+1}\right\rfloor $ and if one of the conditions in Theorem  \ref{teostab0} is satisfied then
\begin{enumerate}
\item  $Hilb^{P(t),s}_{Ch}\ne \emptyset$, and is a regular component of the Hilbert scheme $Hilb_{P(t)}{{\mathbb{P}^n}}$.
\item $\Gamma (\mathcal{B}_d^n)\subset Hilb^{P(t),s}_{Ch},$
\item $\dim Hilb^{P(t),s}_{Ch}=
3g-3+\rho(g,d,n+1)+n(n+2),$
\item $Hilb^{P(t),s}_{Ch}$ is smooth at $\Gamma (z)$ for any
$z \in \mathcal{B}_{d}^{n}$,
\item $\dim Hilb^{P(t),s}_{Ch}/SL(n+1)= 
3g-3+\rho(g,d,n+1)$ and
\item  $\Gamma :\mathcal{B}_{d}^{n}\longrightarrow  Hilb^{P(t),s}_{Ch}$ is an open immersion.
\end{enumerate}
Moreover, the quotient stack $[\Gamma (\mathcal{B}_{d}^{n})/SL(n+1) ]$ is a smooth irreducible Deligne-Mumford stack of dimension
$3g-3+\rho(g,d,n+1)$.
\end{theorem}

The theorem gains in interest if we recall that
the birational geometry of moduli spaces has been a topic of research interest since some time ago.
In particular, the general problem of understanding birational models of the moduli space $M_g$ has received much attention
over the past decade. The goal of
the so called Hassett-Keel program (see \cite{hh}) is to find the minimal
model of the moduli space $M_g$ of curves
via the successive constructions of modular birational
models of $\overline{M_g}$, and aims to give modular interpretations of certain log
canonical models of $\overline{M_g}$. Recall that the most successful approach so
far has been to compare these log canonical models to alternate compactifications
of $\overline{M_g}$ constructed via GIT on the so-called $m$-th Hilbert spaces $Hilb_{P(t)}{{\mathbb{P}^n}}$,
of $m$-canonically embedded curves of genus $g$, for small $m$ and degree $d$.
For a recent account of the
theory we refer the reader to  \cite{jarod} and \cite{melo}.
From the above theorem we have that for low degrees,
up to $d\geq g+ n-\left\lfloor\frac{g}{n+1}\right\rfloor ,$ the Petri curves are in
 any $Hilb^{P(t),s}_{Ch}$, and the bound is sharp.
It  may be possible that could be  easier to describe $[Hilb^{P(t),s}_{Ch}/SL(n+1)]$
 rather than the loci of $m$-canonically
embedded curves, but we will not develop this point here.

We finish this article recalling that in \cite{geis} Gieseker introduced the concept of Hilbert stability for projective curves
and Morrison in \cite[Corollary 3.5 (i)]{mor}, proved that Chow stability implies Hilbert stability.
Therefore,  we can reformulate the above results in terms of Hilbert stability (see Section $4$).

This article is organized as follows. In Section  $2$, we review
some of the standard facts on Chow stability
and establish the relation between Chow stability and stability of
the restriction of the tangent bundle.
In Section  $3$ we summarize the relevant
 material on the stability of $T\mathbb P_{|C} ^n$ and prove
 Theorem \ref{teostab0}. The main results,  Theorems \ref{teochow11} and \ref{teochowprin1} and Corollary \ref{corred1},
 are proved in the fourth section.

{\it Notation:} Given a vector bundle $E$ over $C$ we denote by
$d_E$ the degree, by $n_E$ the rank and by $\mu
(E):=\frac{d_E}{n_E}$ the slope of $E.$  For abbreviation, we
write $H^i(E)$ instead of $H^i(C,E),$ whenever it is convenient.

{\small Acknowledgments: We would like to thank C. Ciliberto, Margarida Melo, J. Alper, A. Beauville and E. Mistretta
 for very helpful conversations and
CONACYT for support. Our special thanks to M. Brion who helped us to make clear some points and specially
to  improve Theorem \ref{teochowprin1}.
 The first author is a member of the research group VBAC
and thanks U. Bhosle, P.E. Newstead and E. Sernesi for their
comments and suggestions. She also thanks ICTP, Trieste for the hospitality as well as for
the support during the preparation of this work. }

\section{Chow stability }

In this section we recall from \cite{mum} the definition of Chow
stability and linear stability for  projective varieties $X\subset
\mathbb{P}^n$ and prove Theorem \ref{teotan}.

Let $X\subset \mathbb{P}^n$ be a non-degenerate irreducible
complex projective variety of dimension $r\geq 1$ and degree
$d\geq 2.$  The Chow form $F _X$ associated to $X\subset
\mathbb{P}^n$ is defined as follows (see \cite{mum}).

Consider the locus $Y_X\subset
\mathbb{G}(n-r-1, \mathbb{P}^n)$ defined by
 $$Y_X:= \{ H\in \mathbb{G}(n-r-1, \mathbb{P}^n) : H\cap X\ne \emptyset \}.$$
 It is well known that $Y_X$ is an irreducible divisor in $\mathbb{G}(n-r-1, \mathbb{P}^n)$
of degree $d$ (in the Pl\"ucker coordinates). Moreover, $Y_X$ is the zero set of a section $F_X \in
H^0(\mathbb{G}(n-r-1, \mathbb{P}^n), \mathcal{O}(d))$ and $F_X$ is
determined up to multiplicative constants. Therefore, it defines a
point $$[F^*_X]\in \mathbb{P}(H^0(\mathbb{G}(n-r-1, \mathbb{P}^n), \mathcal{O}(d))^*).$$
We call $[F^*_X]$ {\it the Chow form} of
$X\subset \mathbb{P}^n.$
 There is a natural action of $SL(n+1)$ on
 $\mathbb{P}(H^0(\mathbb{G}(n-r-1, \mathbb{P}^n), \mathcal{O}(d))^*)$ and a
 GIT concept of $SL(n+1)$-stability (semistability and polystability) for the Chow forms.

 \begin{definition}\label{defchow}\begin{em}  A projective irreducible variety
$X\subset \mathbb{P}^n$ is  Chow stable if the Chow form $[F_X]$ is
$SL(n+1)$-stable. Similarly we define  Chow  semistability
and Chow polystability.
\end{em} \end{definition}

A non-degenerate
complex projective variety  $X\subset \mathbb{P}^n$ defines a point $[X\subset \mathbb{P}^n]$ in a Hilbert
scheme $Hilb_{P(t)}{{\mathbb{P}^n}}$ of $\mathbb{P}^n$, with a suitable Hilbert polynomial $P(t).$ Denote by
$Hilb^{P(t),s}_{Ch}$ (respectively $Hilb^{P(t),ss}_{Ch}$) the irreducible component of the
Hilbert scheme $Hilb_{P(t)}{{\mathbb{P}^n}}$
containing the generic Chow stable (respectively Chow semistable) variety.
Recall that the set of stable Chow forms is irreducible and open in $\mathbb{P}(H^0(\mathbb{G}(n-r-1, \mathbb{P}^n), \mathcal{O}(d))^*).$
Our aim is to describe an open set of $Hilb^{P(t),s}_{Ch}$ when $P(t)=dt+1-g.$

Let $X\subset \mathbb{P}^n$ be as above. Let $(\mathcal{L},V)$ be the generated
linear series on $X$ defining the embedding.  That is,
$\mathcal{L}=\mathcal{O}_X(1)$ is a line bundle on $X$ and  $$V= H^0(\mathbb{P}^n,
\mathcal{O}_{{\mathbb{P}^n}}(1))\subseteq H^0(X,\mathcal{L})$$ is
a linear subspace of sections of dimension $n+1$ which generates
$\mathcal{L}$ and induces the closed immersion
 $$\phi _{\mathcal{L},V}:X\to \mathbb{P}(V^*)=\mathbb{P}^n.$$ Actually,
 $\mathcal{L}= \phi _{\mathcal{L},V}^*(\mathcal{O}_{{\mathbb{P}^n}}(1))$. Note that the
 embedding need not be a canonical embedding, neither $(\mathcal{L},V)$ a complete
 linear system.

We now introduce the notion of linear stability, following Mumford
\cite{mum}. Given a linear space $B \subset \mathbb{P}^n$ of
dimension $n-m-1$ denote by $\pi _ B:\mathbb{P}^n-B\to \mathbb{P}^m$
the canonical projection and by $[\pi _{B}(X)]$ the image cycle of $X$ under
 $\pi _{B}$. That is, $\pi _{B}(X)$ with the multiplicity equal to the
 degree of the $\pi _{B}$ over $\pi _{B}(X)$ if $\dim \pi _{B}(X)=\dim X$, and $0$ otherwise.

 \begin{definition}\label{dells}\cite[Definition 2.16]{mum}.
 \begin{em} If  $X\subset \mathbb{P}^n$
is embedded by the linear series $(\mathcal{L},V)$, we say
that $X \subset \mathbb{P}^n$ (or $(\mathcal{L},V)$)  is
linearly stable (respectively  linearly semistable)
if for all linear subspaces  $W\subset V$
\begin{equation}\label{eqls}
\frac{\deg [\pi _{B}(X)]}{\dim W -r} >  \frac{\deg
{X}}{n+1-r}   \ \ \  (respectively \geq ),
 \end{equation}
 where $B= \mathbb{P}(Ann W)$ and $\dim \pi _{B}(X)=\dim X$ .
\end{em} \end{definition}

 \begin{remark}\label{remsing}\begin{em}  Recall from \cite[Proposition 2.5]{mum}
that
if we project from a point $p\in X$ then
$$\deg [\pi _{p}(X)]= \deg X-\mu _p,$$ where  $\mu _p $ is the multiplicity of $p\in X.$
 \end{em}
 \end{remark}

 Here is an elementary property of these concepts.

\begin{proposition}\label{proplss} Let
$X\subset \mathbb{P}^n $ be a projective variety of dimension $r\geq 1$ and $p$ a point in $X$. If
$X\subset \mathbb{P}^n$ is linearly stable (respectively semistable) then $\mu _p < \frac{\deg X}{n+1-r}$ (respectively $\leq $). In particular, if $X$ is a linearly stable (respectively semistable) curve
then $\mu _p < (\leq )\frac{\deg X}{n}.$
\end{proposition}

\begin{proof} Suppose that $V$ is a vector space of dimension $n+1$ and $\mathbb{P}^n=\mathbb{P}^n(V)$.
 Assume that $X\subset \mathbb{P}^n(V)$ is linearly stable
and let $W\subset V$ be a subspace of dimension $n.$
The following inequality follows from $(\ref{eqls})$ and Remark \ref{remsing}
\begin{equation}\label{eqlss}
\frac{\deg X-\mu _p}{n -r} >  \frac{\deg
{X}}{n+1-r}.
 \end{equation}
Hence, from $(\ref{eqlss})$, $\mu _p < \frac{\deg X}{n+1-r}$, which is the desired conclusion.
 \end{proof}

The relation between Chow stability and linear stability for
curves is established by the next theorem, which goes back as far
as \cite{mum}.

\begin{theorem}\label{teocls}\cite[Theorem 4.12]{mum} Let $C\subset
\mathbb{P}^n$ be a curve. If $C\subset \mathbb{P}^n$
is linearly stable (respectively semistable) then $C\subset
\mathbb{P}^n$ is Chow stable (respectively Chow semistable).
\end{theorem}

In the remainder of this section we assume $X$ to be an
irreducible smooth curve $C$ of genus $g>1$.

Let $(L,V)$ be a
generated linear series of type $(d,n+1)$ over $C$, that is,
the degree of $L$ is $d$ and $\dim V= n+1.$
Recall that the
pull-back, by $\phi _{L,V}$, of the dual of the
Euler sequence tensored by $ \mathcal{O}_{{\mathbb{P}^n}}(1)$
induces the following exact sequence
\begin{equation}\label{eq1}
0\to M_{L,V}\to V\otimes \mathcal{O}_C \to L
\to 0
\end{equation}
of vector bundles over $C$. The kernel $M_{L,V}$ of the
evaluation map  $V\otimes \mathcal{O}_C\to L$ is called
the {\it $(1^{st}-)$syzygy bundle} of $(L,V)$ (sometimes also a Lazarsfeld or Lazarsfeld-Mukai bundle or the dual span bundle). If
$V=H^0(C,L),$ we denote $ M_{L,V}$ by $
M_{L}$ and $\phi _{L,V}$ by $\phi
_{L}.$

In general, the syzygy bundles can be defined for any
variety $X$ and arise in a variety of algebraic and
geometric problems. For curves, the stability of $M^*_{L,V}$  has been
related to problems like: Green's Conjecture,
the Minimal Resolution Conjecture, computing Koszul cohomology groups,
theta-divisors, Brill-Noether theory and coherent system theory.
 In some sense, they have the information on how
complicated $X$ sits in $\mathbb{P}^n$. We use it now to prove Chow stability.

We can now prove Theorem \ref{teotan}

\bigskip

{\it Proof of Theorem \ref{teotan} } The proof is based on the following observation. For complex
irreducible smooth curves the
stability of $T\mathbb P_{|C} ^n$ is
equivalent to the stability of the syzygy bundle
$M_{L,V}$ since
\begin{equation}\label{eq0}
T\mathbb P_{|C} ^n= M_{L,V}^*\otimes
L.
\end{equation}

By assumption, $T\mathbb P_{|C} ^n$ is stable, hence $M_{L,V}$ is stable.

Recall that the stability of $M_{L,V}$ gives the following inequality
 \begin{equation}\label{eq2}
 \frac{\deg (F)}{rk \ F} <
 \frac{\deg (M_{L,V})}{rk \ M_{L,V}},\end{equation}
 for every proper subbundle $F\subset M_{L,V}.$
 In particular, for those subbundles that fit into the following diagram
 \begin{equation}\label{diag1}
\begin{array}{ccccccccc}
&& 0&& 0& &0&& \\
&& \downarrow&& \downarrow& &\downarrow&& \\
0&\rightarrow &M_{{L}', W}&\rightarrow& W\otimes
\mathcal{O}_C
& \rightarrow &{L}'&\rightarrow &0,\\
&& \downarrow&& \downarrow& &\downarrow&& \\
0&\rightarrow &M_{L,V}&\rightarrow& V\otimes
\mathcal{O}_C
& \rightarrow &{L}&\rightarrow &0\\
\end{array}
\end{equation}
 where $W\subset V$ is a linear subspace that generates the line bundle ${L}'$.
 Applying $(\ref{eq2})$ we deduce that
 \begin{equation}\label{eqdes}
 \frac{-\deg (L')}{\dim W -1}<\frac{-\deg (L)}{n}.
 \end{equation}

 Since $\deg C = \deg (L)$ and $\deg [(p_B(C)]=\deg (L')$ when $B= \mathbb{P}(Ann W)$,
$(\ref{eqdes})$ shows that $C \subset \mathbb{P}^n$ is linearly stable, by $(\ref{eqls})$,
and finally, from Theorem \ref{teocls}, that $C \subset \mathbb{P}^n$ is Chow stable,
and this is precisely the assertion of the theorem.
Similarly, we obtain Chow semistability.

$ \hfill{\Box} $

\begin{remark}\begin{em}\footnote{The relation of stability of $M_L$ with linear stability was also
proved recently in \cite{bstop}.}
The implications: Chow stability implies linear stability or linear stability implies stability of $M_{L,V}$ are not true in general. In \cite{ms} was proved the equivalence of linear stability and stability of $M_{L,V}$ for certain bounds given by the  Clifford index  of the curve C, and in  \cite{hugo} the second author prove that, for general curves, linear stability implies stability of $M_{L,V}$ if $(n-1)codim V<g$.
\end{em}\end{remark}

\begin{corollary}\label{corram} If $C$ is a non-hyperelliptic smooth curve then $C\subset \mathbb{P}(H^0(C,K_C)^*)$ is Chow stable.
\end{corollary}

\begin{proof}
In \cite{ram} Paranjape and Ramanan proved that $M_{{K_C}}$ is stable if $C$ is non-hyperelliptic. Theorem \ref{teotan} now shows that $C\subset \mathbb{P}(H^0(C,K_C)^*)$ is Chow stable.
\end{proof}

In the next section we summarize the known cases were $T\mathbb P_{|C} ^n$ is stable.

 \section{Stability of $T\mathbb P_{|C} ^n$ }

Let $C$ be a complex
irreducible smooth curve of genus $g\geq 1$.
In this section we summarize without proofs the relevant material in
the stability of the syzygy bundles over $C$. For the proofs we refer the
reader to e.g. \cite{ln}, \cite{mio}, \cite{bbn1}, \cite{bstop} and \cite{bbn}.

It is well known (\cite[Proposition 3.2]{ein}) that for  smooth irreducible curves and
complete linear system, $M_L$ is stable if $d>2g$
and semistable if $d=2g$.
The reference \cite{came} includes an example to show
that these nice results for any curve do not extend beyond $d\geq 2g.$
For complete linear systems and general curves of lower degrees, $M_L$ is always semistable
(see \cite{sh}) and conditions for stability are given in
\cite[Theorem 2]{bu} (see also \cite{mio} and \cite{bbn1}).
The stability of $M_L$ was proved also for
\begin{enumerate}
\item some line bundles on special curves and on curves computing the
Clifford dimension (\cite{la}),
\item linear systems computing the Clifford index (\cite{la} and \cite{ms}),
\item line bundles with degree bounded with the Clifford index (\cite{came} and \cite{ms}).
\end{enumerate}

A breakthrough result has been achieved recently in \cite{bbn},
where the semistability of the syzygy bundle for non complete linear
series was proved and stability under some conditions for general and Petri curves.
Recall that a  smooth curve is called {\it Petri }
if for every line bundle $L$ on $C$, the cup product map
\begin{equation}\label{eqpetri} \mu : H^0(C,L)\otimes H^0(C,L^*\otimes K_C) \rightarrow H^0(C,K_C)
\end{equation}
is injective.

 The next corollary is a reformulation of the main results on
 syzygy bundles in terms of Chow stability for smooth irreducible curves. For a deeper discussion of the stability of the syzygy bundles we refer the reader to \cite{bbn}.

\begin{corollary}\label{teolow} Let $(L,V)$ be a generated linear series
of type $(d,n+1)$ over an irreducible smooth curve $C$ of genus $g\geq 1.$ Assume $C$ and $(L,V)$
are general and $d\ge g+n-\left\lfloor\frac{g}{n+1}\right\rfloor$. Then $\phi _{L,V}(C)\subset \mathbb{P}^n$ is Chow semistable.
Moreover, $\phi _{L,V}(C)\subset \mathbb{P}^n$ is Chow stable if one of the following conditions
\begin{enumerate}
\item $C$ is a general curve of genus $g\geq 2$ and $gcd(d,n)=1$;
\item $C$ is a curve of genus $g=1, d\geq n+1$ and $gcd(d,n)=1;$
\item $C$ is a curve of genus $g=2, d\geq n+2$ with $d\ne 2n$;
\item $C$ is a Petri curve of genus $g\geq 3$ and $n\leq 4;$
\item $C$ is a Petri curve of genus $g\geq 2(n-2)$ and $n \geq 5;$
\end{enumerate}
is satisfied.
\end{corollary}

\begin{proof} From Theorem \ref{teotan} we only need to show that, under the above hypothesis,
the syzygy bundle $M_{L,V}$ is semistable.

The semistability of $M_{L,V}$ was proved in \cite[Theorem 5.1]{bbn} for a general $C$ and $(L,V)$.
(see also \cite{sh} for complete linear systems). Case $(1)$ follows immediately from the equality $gcd(d,n)=1$.

For $(2),(3)$, the stability
of $M_{L,V}$ follows from \cite{ln}, if $C$ is a curve of genus
$g=1, d\geq n+1$ and $gcd(d,n)=1$, and from \cite[Proposition 6.5]{mio} and
\cite[Theorem 8.2]{bbn1}, if $C$ is a curve of genus $g=2, d\geq n+2$ with $d\ne 2n$.

Let  $C$ be a Petri curve of genus $g\geq 3.$
The stability of $M_{L,V}$ was proved in \cite[\S 7]{bbn1}
when  $n\leq 4$ and in \cite[Theorem 6.1.]{bbn} when $n\geq 5$ and $g\geq 2(n-2).$

From what has already been proved (see Theorem \ref{teotan})
it follows that $\phi _{L,V}(C)\subset \mathbb{P}^n$ is Chow stable, which is the desired conclusion.
\end{proof}

\begin{remark}\label{remsmall}\begin{em} Note that in the above cases
we can have $C\subset \mathbb{P}^n$ with $n\ne d-g.$
\end{em}\end{remark}

\section{The Hilbert Scheme}

Let $Hilb_{P(t)}{{\mathbb{P}^n}}$ be the Hilbert scheme of
$\mathbb{P}^n$ with Hilbert polynomial
$P(t)= dt+1-g$. Recall that $Hilb^{P(t),s}_{Ch}$ is the irreducible component of $Hilb^{P(t)}_{{\mathbb{P}^n}}$ containing
the generic $[C \subset \mathbb{P}^n ] \in Hilb^{P(t)}_{{\mathbb{P}^n}}$
such that the Chow form $[F_C]$ is Chow stable. Respectively $Hilb^{P(t),ss}_{Ch}$ is the component of Chow semistable curves.
A component of $Hilb_{P(t)}{{\mathbb{P}^n}}$ is said to be regular if its general point corresponds to
a smooth irreducible curve $C\subset \mathbb P$  with $H^1(C,N_{{C/\mathbb{P}^n}})=0.$

Fix $P(t)=dt+1-g$. From the Brill-Noether theory we know that
there are no general curves in $Hilb_{P(t)}{{\mathbb{P}^n}}$ if $d< g+n - \left\lfloor\frac{g}{n+1}\right\rfloor$ and only a finite number (the Castelnuovo number) of general curves if $d=g+n - \left\lfloor\frac{g}{n+1}\right\rfloor .$ Moreover, from \cite{la} and \cite{ms} and Theorem \ref{teotan} such curves are Chow semistable.
In this section we
are interested in describing a smooth open set of
$Hilb^{P(t),s}_{Ch}$ for $d-g>n-\left\lfloor\frac{g}{n+1}\right\rfloor$.
For a treatment of the case $n= d-g$ and $d>>0$  we refer the reader to \cite{melo}.

We introduce the notions of the basic varieties of the Brill-Noether theory for moving curves, following \cite[Chapter XXI]{arb2}.

Let $p:\mathcal{C}\to S$ be a family of smooth curves of genus $g\geq 2$ parameterized by a scheme S.
A family of $g^n_d$'s on $p:\mathcal{C}\to S$ parameterized by an $S$-scheme $f:T\to S$ is a pair $(\mathcal{L},\mathcal{V})$, where $\mathcal{L}$ is a line bundle over $\mathcal{C} \times _S T$ whose restriction $\mathcal{L}_t$ to each fibre of $\mathcal{C} \times _S T\to T$ has degree $d$, and $\mathcal{V}$ is a locally free subsheaf of $(p_{T})_*(\mathcal{L})$ of rank $n+1$ such that, for each $t\in T$, the restriction $\mathcal{V}_t$ is a subspace of $H^0(p^{-1}(t), \mathcal{L}_t)$. Two families $(\mathcal{L},\mathcal{V})$ and $(\mathcal{L}',\mathcal{V}')$ are equivalent if there is a line bundle $\mathcal{N}$ on $T$ and an isomorphism $k:\mathcal{L}\stackrel{\cong}{\to}\mathcal{L}'\otimes\mathcal{N}$ which induces an isomorphism $k_*:\mathcal{V}\stackrel{\cong}{\to}\mathcal{V}'\otimes\mathcal{N}.$

\begin{remark}\label{remeqfam}\begin{em} Fix a projective space $\mathbb{P}^n$. Let $[(\mathcal{L},\mathcal{V})]$ be an equivalence
class of $g^n_d$'s families on $p:\mathcal{C}\to S$ parameterized by an $S$-scheme $f:T\to S$. For any
 $(\mathcal{L},\mathcal{V})\in [(\mathcal{L},\mathcal{V})]$
 the vector bundle associated to the locally free subsheaf $\mathcal{V}$ defines a  $PGL_n$-principal bundle $\mathcal{B}_d^n(V)$  with fibre
$ \{\alpha :\mathbb{P}(V^*_t)\to \mathbb{P}^n: \alpha \ \ \mbox{is an isomorphism} \}$.
 This $PGL_n$-principal bundle is independent of which member of $[(\mathcal{L},\mathcal{V})]$ we choose to define it. Indeed, for any other $(\mathcal{L}',\mathcal{V}') \in [(\mathcal{L},\mathcal{V})]$, $\mathcal{V}\stackrel{\cong}{\to}\mathcal{V}'\otimes\mathcal{N}$ and hence the $PGL_n$-principal bundles $\mathcal{B}_d^n(V)$  and $\mathcal{B}_d^n(V')$  are the same.
In this way we obtain what will be referred to as the $PGL_n$-principal bundle of $[(\mathcal{L},\mathcal{V})]$, and will be denoted by $g:\mathcal{B}_d^n[(\mathcal{L},\mathcal{V})]\to T$.
\end{em}\end{remark}

Suppose now that $p:\mathcal{C}\to S$ admits a section. From \cite[Chapter XXI, Theorem 3.13]{arb2} there exists
an $S$-scheme $h:\mathcal{G}_d^n(p)\to S$ representing the functor

 \[
 \begin{matrix}
  G_d^n(p):Sch&\longrightarrow& Sets \hspace{4.5cm}  \\
T&\longmapsto &\left\{\begin{smallmatrix}
\mbox{equivalence classes of families of} \ \   g^{n}_{d} \  's \ \ \mbox{on} \\
p: \ \mathcal{C}\to S \  \  \mbox{parametrized by} \ \ T \hspace{2.5cm}
\end{smallmatrix}\right\}.
\end{matrix}
\]

\bigskip
Set theoretically $$supp(\mathcal{G}_d^n(p))= \{(s,(L,V)): s\in S, \ (L,V)\in G^n_d(C_s) \},$$
where  $G_d^n(C_s)$ is the variety
of linear series $(L,V)$ of type $(d,n+1)$ on $C_s.$
Let $(\mathcal{L},\mathcal{U})$ be the universal family of $g^{n}_{d}$'s on $p:\mathcal{C}\to S$ parameterized by $\mathcal{G}_d^n(p)$. For simplicity of notation,  we
write $g:\mathcal{B}_d^n(p)\to \mathcal{G}_d^n(p) $ instead of $g:\mathcal{B}_d^n[(\mathcal{L},\mathcal{V})]\to \mathcal{G}_d^n(p)$, when no confusion can arise.
We will represent any element $z\in \mathcal{B}_d^n(p)$ by the triple
$$z:=(C,(L,V), \alpha : \mathbb{P}(V^*)\to \mathbb{P}^n )$$with $(L,V)\in G_d^n(C)$ where $C=C_s$ and $h(g(z))=s$.

Denote by $M_g$ the moduli space of smooth curves of genus $g>2$ and by $M^0_g$
the moduli space of automorphisms-free smooth curves.
We follow  \cite{hm}, or \cite{arb2}, in assuming that $M^0_g$ is a fine moduli space. Thus,
there exists a universal family $p_0:\mathcal{C}\to M^0_g$. However, $p_0:\mathcal{C}\to M^0_g$ admits no section, hence, there is no universal family $(\mathcal{L},\mathcal{V})$. Recall that the functor $G_d^n(p_0)$ is representable only when the family has a section. Nevertheless, we will now show how to dispense with the assumption on the existence of a global section.

This theorem ensures the existence of the $PGL_n-$principal bundle $g:\mathcal{B}_d^n(p_0) \to \mathcal{G}_d^n(p_0)$,
and may be proved in much the same way as the existence of $\mathcal{G}_d^n(p_0)$ (see \cite[Chapter XXI, \S 3]{arb2}]). This construction is adapted from \cite[Chapter XXI, \S 2 and 3]{arb2} and we give only the main ideas of the proof.

\begin{theorem}\label{teoprinbund}   There exists a scheme $\mathcal{G}_d^n(p_0)$ parameterizing all
$ g^{n}_{d} $'s on the fibres of $p_0:\mathcal{C}\to M^0_g.$ Moreover, there exists a $PGL_n$-principal bundle $g:\mathcal{B}_d^n(p_0) \to \mathcal{G}_d^n(p_0) $
with fibre $ \{\alpha :\mathbb{P}(V^*)\to \mathbb{P}^n: \alpha \in PGL_n  \}$ at $w=(s,(L,V))\in \mathcal{G}_d^n(p_0).$
\end{theorem}

\begin{proof} It suffices to make the following observation. We can cover $M^0_g$ with open sets $\{U_\alpha\}$ where a section of $p|_{U_\alpha}:=(p_0)|_{U_\alpha}:\mathcal{C}_{U_\alpha}\to U_\alpha$ exists. Recall that a local Kuranishi family of smooth curves admits an analytic section. Thus, the functor $G_d^n(p|_{U_\alpha})$ is representable by $\mathcal{G}_d^n(p|_{U_\alpha})$ with  $(\mathcal{L}(U_\alpha),\mathcal{U}(U_\alpha))$ as the universal family of $g^{n}_{d}$'s on $p|_{U_\alpha}:\mathcal{C}_{U_\alpha}\to U_\alpha$ parameterized by $\mathcal{G}_d^n(p|_{U_\alpha})$. One then constructs $\mathcal{G}_d^n(p|_{U_\alpha})$ and uses the universal properties to patch together the $\mathcal{G}_d^n(p|_{U_\alpha})$'s to define an analytic variety $\mathcal{G}_d^n(p_0)$. To prove the existence of $\mathcal{G}_d^n(p_0)$ in the category of schemes, use a Zariski-open neighborhood $U_\alpha$ for each point in $M^0_g$ and a finite \'etale
 base change $f:U_\alpha '\to U_\alpha$ such that the family $p'|_{U_\alpha}:f^*(\mathcal{C}_{U_\alpha})\to U_\alpha '$ has a section. For each $U_\alpha$ there is a natural projection map of the analytic spaces $\mathcal{G}_d^n(p'|_{U_\alpha '})\to \mathcal{G}_d^n(p|_{U_\alpha} )$. From \cite[Chapter XXI, Lemma 2.12]{arb2}, the $\mathcal{G}_d^n(p|_{U_\alpha})$'s have a scheme structure which, using the universal properties, patch together to define a scheme structure in  $\mathcal{G}_d^n(p_0)$.

 What is left is to show that exists the $PGL_n$-principal bundle. Since the functor $G_d^n(p|_{U_\alpha})$ is representable  for the family $p|_{U_\alpha}:=(p_0)|_{U_\alpha}:\mathcal{C}_{U_\alpha}\to U_\alpha$,
  one  constructs the $PGL_n$-principal bundle $\mathcal{B}_d^n(p|_{U_\alpha}) \to \mathcal{G}_d^n(p|_{U_\alpha}) $ over $\mathcal{G}_d^n(p|_{U_\alpha})$ from the universal family $(\mathcal{L}({U_\alpha}),\mathcal{U}({U_\alpha}))$. From Remark \ref{remeqfam}, are well defined, and one uses the universal properties to patch them together. By a similar argument, the existence of $g:\mathcal{B}_d^n(p_0) \to \mathcal{G}_d^n(p_0) $ in the category of schemes follows from  \cite[Chapter XXI, Lemma 2.12]{arb2}.
  \end{proof}

An important open
sublocus  of $M^0_g$ is the one whose points represent Petri curves without automorphisms.
It is well known that for Petri curves $C$,
$G_d^n(C)$ is empty if the Brill-Noether number $ \rho (g,d,n+1):= g-(n+1)(n-d+g)$
is negative and is irreducible and smooth of dimension $\rho (g,d,n+1)$ if $\rho (g,d,n+1) > 0.$
The following Proposition may be proved in much the same way as the above results.

\begin{proposition}\label{propg}(\cite[Proposition 5.26 and Corollary 5.30]{arb2})
$\mathcal{G}_d^n (p_0) $ is empty if
the Brill-Noether number $\rho (g,d,n+1)= g-(n+1)(n-d+g)$ is negative
and is irreducible and smooth of dimension $3g-3 +\rho$ if $\rho (g,d,n+1)> 0.$
\end{proposition}

Let us denote by  $\mathcal{P}_d^n(p_0) \subset \mathcal{G}_d^n(p_0)$  the
subscheme of general linear series over Petri curves without automorphisms.

\begin{corollary}\label{corpetri} Let $g\geq 4$. If $d-g>n-\left\lfloor\frac{g}{n+1}\right\rfloor$ then $\mathcal{P}_d^n(p_0) $ is open subscheme of
$\mathcal{G}_d^n (p_0)$.
 Moreover, $\mathcal{P}_d^n(p_0) $
is irreducible and smooth of dimension $3g-3 +\rho$.
\end{corollary}

 \begin{proof} The proof is straightforward.
 \end{proof}

Denote by $g:\mathcal{B}_d^n \to \mathcal{P}_d^n(p_0)$ the restriction of the $\mathcal{B}_d^n(p_0) \to \mathcal{G}_d^n(p_0) $ to $\mathcal{P}_d^n(p_0) $.
That is,
set theoretically,
$$supp(\mathcal{B}_d^n)=\{(s,(L,V),\alpha :\mathbb{P}(V^*)\to \mathbb{P}^n ) :s\in S, \ \mbox{and} \hspace{6cm}  $$
$$ (L,V)\in G_d^n(C_s) \  \mbox{where} \ \phi _{L,V}\ \mbox{is an embedding}  \ \mbox{and} \ C_s \ \mbox{is Petri without automorphisms} \}.$$

The relation between $\mathcal{B}_d^n $ and $Hilb^{P(t),s}_{Ch}$ is established by our next Theorem.

\begin{theorem}\label{teochow11} Let $g,d$ and $n$ be positive integers such that
$g\geq 4$ and $d-g> n-\left\lfloor\frac{g}{n+1}\right\rfloor .$ There exists a natural injective morphism $\Gamma: \mathcal{B}_d^n \to Hilb_{P(t)}{{\mathbb{P}^n}},$ such that
$\Gamma (\mathcal{B}_d^n)\subset Hilb^{P(t),ss}_{Ch}.$ Moreover,
 $Hilb^{P(t),s}_{Ch}\ne \emptyset$ if one of the conditions in Theorem \ref{teolow} is satisfied.
\end{theorem}

\begin{proof} Under the conditions stated above, $\mathcal{B}_d^n \ne \emptyset$. The incidence variety
$$\mathcal{F}:= \{ (z,t) \in \mathcal{B}_d^n \times {{\mathbb{P}^n}}: t\in \alpha(\phi_{L,V}(C)) \ \mbox{if}\  z=(C,(L,V), \alpha : \mathbb{P}(V^*)\to \mathbb{P}^n )\}$$ and the projection $\mathcal{F}\to \mathcal{B}_d^n $
defines a family $\mathcal{F} \subset \mathcal{B}_d^n \times {{\mathbb{P}^n}}$ of curves in $\mathbb{P}^n$, parameterized by $\mathcal{B}_d^n$.
From this, and the definition of $\mathcal{B}_d^n$, it follows that there exists an injective morphism
$$\Gamma: \mathcal{B}_d^n \to Hilb_{P(t)}{{\mathbb{P}^n}},$$
$$z=(C,(L,V), \alpha : \mathbb{P}(V^*)\to \mathbb{P}^n )\longmapsto [\alpha(\phi_{L,V}(C))\subset  \mathbb{P}^n].$$

Theorem \ref{teolow} shows that, for general $(L,V)$,
$\phi_{L,V}(C)\subset  \mathbb{P}^n$ is Chow semistable, hence $Hilb^{P(t),ss}_{Ch}\ne \emptyset$,
Similarly, $Hilb^{P(t),s}_{Ch}\ne \emptyset $ under the conditions of Theorem \ref{teolow}.
From the first part and the definition of $\mathcal{B}_d^n$ we see that
$\Gamma (\mathcal{B}_d^n)\subset Hilb^{P(t),s}_{Ch}$ as claimed.
\end{proof}

This theorem yields information about the structure of $Hilb^{P(t),s}_{Ch}$.

\begin{theorem}\label{teochowprin1} Under the conditions of Theorem \ref{teochow11}
\begin{enumerate}
\item $\dim Hilb^{P(t),s}_{Ch}=
3g-3+\rho(g,d,n+1)+n(n+2),$
\item $Hilb^{P(t),s}_{Ch}$ is smooth at $\Gamma (z),$ for any
$z \in \mathcal{B}_{d}^{n}$ and
\item $\dim Hilb^{P(t),s}_{Ch}/SL(n+1)= \dim \mathcal{P}_d^n =
3g-3+\rho(g,d,n+1)$.
\end{enumerate}
Moreover, if $d-g> n-\left\lfloor\frac{g}{n+1}\right\rfloor ,$ $Hilb^{P(t),s}_{Ch}$ is regular component of the Hilbert scheme $Hilb_{P(t)}{{\mathbb{P}^n}}$
 and $\Gamma :\mathcal{B}_{d}^{n}\longrightarrow  Hilb^{P(t),s}_{Ch}$ is an open immersion when  $g\geq 4$.
\end{theorem}

\begin{proof} From Corollary \ref{corpetri}, $\mathcal{B}_{d}^{n}$ is smooth of dimension $3g-3+\rho(g,d,n+1)+n(n+2)$ and irreducible if $d-g> n-\left\lfloor\frac{g}{n+1}\right\rfloor .$

 We claim that
$Hilb^{P(t),s}_{Ch}$ is smooth at
$\Gamma (z)=[\alpha(\phi _{L,V}(C))\subset \mathbb{P}^n]$ of dimension
$\chi (N_{{C/\mathbb{P}^n}})$, where $N_{{C/\mathbb{P}^n}}$ is the normal bundle of $C$ in $\mathbb{P}^n$.
To see this, it suffices to
show that $h^0(C,N_{{C/\mathbb{P}^n}})= \chi (N_{{C/\mathbb{P}^n}})$ or, equivalently, that
$h^1(C,N_{{C/\mathbb{P}^n}})=0.$

Recall that the normal bundle fits into the following exact sequence
\begin{equation}\label{eqnor}
0\to TC\to T\mathbb P_{|C} ^n \to N_{{C/\mathbb{P}^n}} \to 0
\end{equation}
of vector bundles over $C$.
Let
\begin{equation}\label{eqcoh}
0\to H^0( TC)\to H^0(T\mathbb P_{|C} ^n) \to H^0(N_{{C/\mathbb{P}^n}}) \to H^1( TC)\to H^1(T\mathbb P_{|C} ^n) \to H^1(N_{{C/\mathbb{P}^n}}) \to 0
\end{equation}
be the cohomology sequence of $(\ref{eqnor})$.

Let us first prove that $h^1(C,T\mathbb P_{|C} ^n)=0.$ From $(\ref{eq0})$ we have
$T\mathbb P_{|C} ^n= M_{L,V}^*\otimes L.$  This, and Serre duality, gives
\begin{equation}\label{eqtand}
h^1(C,T\mathbb P_{|C} ^n)=h^0((T\mathbb P_{|C} ^n)^*\otimes K_C) = h^0(M_{L,V}\otimes L^*\otimes K_C).
\end{equation}
Tensor the exact sequence $(\ref{eq1})$ with $L^*\otimes K_C$ to get the exact sequence
\begin{equation}\label{eqtan}
0\to M_{L,V}\otimes L^*\otimes K_C\to V\otimes \mathcal{O}_C\otimes L^*\otimes K_C \to  K_C
\to 0.
\end{equation}
Since $C$ is a Petri curve, $(\ref{eqpetri})$ shows that  $h^0(C,M_{L,V}\otimes L^*\otimes K_C)=0 $.
 Therefore, from $(\ref{eqtand})$, $h^1(C,T\mathbb P_{|C} ^n)=0.$
Hence, from this and  Riemann-Roch formula we deduce that
\begin{equation}\label{eqnor1}
h^0(C,T\mathbb P_{|C} ^n)=d(n+1)+n(1-g)=\rho(g,d,n+1)+n(n+2).
\end{equation}

Let us now compute the cohomology of the normal bundle.  From what has already been proved and the sequence $(\ref{eqcoh})$ we deduce that
$H^1(C,N_{{C/\mathbb{P}^n}})=0,$ and hence $\chi (N_{{C/\mathbb{P}^n}})=h^0(C,N_{{C/\mathbb{P}^n}})$, as claimed.

As $g\geq 3$, we have $H^0(C, TC)=0$ and $h^1(C, TC)=3g-3.$
We conclude from $(\ref{eqcoh})$ and $(\ref{eqnor1})$ that
$$h^0(C,N_{{C/\mathbb{P}^n}})= 3g-3+\rho(g,d,n+1)+n(n+2),$$ hence that
 $$\dim Hilb^{P(t),s}_{Ch}=h^0(C,N_{{C/\mathbb{P}^n}})= 3g-3+\rho(g,d,n+1)+n(n+2)=\dim \mathcal{B}_{d}^{n}$$ and finally that $Hilb^{P(t),s}_{Ch}$ is smooth at $\Gamma (z)$.
 This clearly forces
 $$\dim Hilb^{P(t),s}_{Ch}/SL(n+1)= 3g-3+\rho(g,d,n+1).$$ The morphism $\Gamma :\mathcal{B}_{d}^{n}\longrightarrow  Hilb^{P(t),s}_{Ch}$ is injective. From \cite[Corollaire (4.4.9)]{grot} we conclude that
$\Gamma :\mathcal{B}_{d}^{n}\longrightarrow  Hilb^{P(t),s}_{Ch}$ is an open immersion, if $d-g> n-\left\lfloor\frac{g}{n+1}\right\rfloor$,  which proves the theorem.

\end{proof}

\begin{corollary}\label{corred1} The quotient stack $[\Gamma (\mathcal{B}_{d}^{n})/SL(n+1) ]$ is a smooth irreducible Deligne-Mumford stack of dimension $3g-3+\rho(g,d,n+1)$.
\end{corollary}

 \begin{proof} It is clear that $\mathcal{B}_{d}^{n}$ is an atlas and, according to the results of the previous theorem, it is a smooth and irreducible scheme, and this is precisely the assertion of the corollary.
 \end{proof}

We finish this section recalling that Morrison in \cite[Corollary 3.5 (i)]{mor} (see also \cite{mab}),
proved that Chow stability implies Hilbert stability (see e.g. \cite{geis}).

Denote by $Hilb^{P(t),s}_{Hilb}$  the irreducible component of the
Hilbert scheme $Hilb^{P(t)}_{{\mathbb{P}^n}}$
containing the generic Hilbert stable curve. From what has already been proved, for a complex
irreducible smooth curve $C \subset \mathbb{P}^n$  we deduce that
\begin{itemize}
\item if  $T\mathbb P_{|C} ^n$
 is stable then $C\subset
\mathbb P ^n$ is Hilbert stable.
\item If  $T\mathbb P_{|C} ^n$
is irreducible and admits an Hermitian-Einstein metric then $C\subset
\mathbb P ^n$ is Hilbert stable.
\item  Under the conditions of Theorem \ref{teolow}
\begin{enumerate}
\item $Hilb^{P(t),s}_{Hilb}\ne \emptyset.$
\item  $\Gamma (\mathcal{B}_d^n)\subset Hilb^{P(t),s}_{Hilb}$.
\item  $Hilb^{P(t),s}_{Hilb}$ has dimension $3g-3+\rho(g,d,n+1)+n(n+2)$ and  is smooth at $\Gamma (z)$ for any
$z \in \mathcal{B}_{d}^{n}.$
\item $\dim Hilb^{P(t),s}_{Hilb}/SL(n+1)= \dim \mathcal{P}_d^n =
3g-3+\rho(g,d,n+1)$.
\end{enumerate}
\end{itemize}


\end{document}